\documentclass[final,5p,times,twocolumn]{elsarticle}

\usepackage[english]{babel}
\usepackage{amsmath,amssymb,mathtools}
\usepackage{amsthm}
\usepackage{enumitem}
\usepackage{xcolor}
\usepackage{subcaption}
\usepackage[colorlinks=true,allcolors=blue]{hyperref}
\usepackage{tikz}
\usepackage{mdframed}

\newcommand{\norm}[1]{\left\lVert #1\right\rVert}

\theoremstyle{plain}
\newtheorem{theorem}{Theorem}
\newtheorem{lemma}[theorem]{Lemma}
\newtheorem{proposition}[theorem]{Proposition}

\theoremstyle{definition}
\newtheorem{assumption}[theorem]{Assumption}
\theoremstyle{remark}
\newtheorem{remark}[theorem]{Remark}

\newcommand{\R}{\mathbb{R}}
\newcommand{\N}{\mathbb{N}}

\journal{European Journal of Control}

\begin{document}
\begin{frontmatter}

\title{Euler Time-Discretization of Funnel-Controlled Dynamics}

\author[Halle]{Thomas Berger}
\author[Wupp,IEG]{Hannes Gernandt}
\author[Halle]{Pushya Mitra}

\affiliation[Halle]{organization={Institute of Mathematics, Martin Luther University Halle-Wittenberg},
                    addressline={Theodor-Lieser-Stra\ss e 5},
                    city={Halle (Saale)},
                    postcode={06120},
                    country={Germany}}

\affiliation[Wupp]{organization={Bergische Universit\"at Wuppertal},
                    addressline={Gau\ss stra\ss e 20},
                    city={Wuppertal},
                    postcode={42119},
                    country={Germany}}

\affiliation[IEG]{organization={Fraunhofer Research Institution for Energy Infrastructures and Geotechnologies IEG},
                    addressline={Gulbener Stra\ss e 23},
                    city={Cottbus},
                    postcode={03046},
                    country={Germany}}

\begin{abstract}
We analyze explicit and implicit Euler discretizations of the continuous-time
closed-loop dynamics generated by a funnel controller for a fully actuated
nonlinear full-state system of relative degree one with scalar positive input
gain.  
A quantitative
continuous-time funnel margin provides the compactness and regularity bounds
needed for the error analysis.  For explicit Euler, a computable upper
step-size threshold guarantees a global nodewise distance from the funnel
boundary. We show that first-order convergence holds on finite time horizons and, under a~dissipativity condition (a~one-sided Lipschitz bound on the plant
dominated by the funnel damping), uniformly in time. For implicit Euler, an admissible update exists for every step size and every trajectory initialized in the funnel retains a step-size-independent nodal margin. Under dissipativity, the admissible root is unique for every step size, the discrete flow is contractive, and the first-order error bound is uniform in time. 
\end{abstract}

\begin{keyword}
Funnel control \sep time discretization \sep explicit Euler \sep implicit Euler
\sep prescribed performance \sep contractivity
\end{keyword}

\end{frontmatter}

\section{Introduction}\label{sec:introduction}
Funnel control is an output-feedback methodology for reference tracking with
prescribed transient performance and tracking accuracy.  Given a reference
signal $y_{\rm ref}$ and a positive time-varying boundary $\psi$, the tracking
error $e=y-y_{\rm ref}$ is required to satisfy
\[
    \forall\, t\ge 0:\ (e(t),t)\in\mathcal F_\psi,\quad
    \mathcal F_\psi
    :=\{(e,t)\in\mathbb R^m\times \mathbb R_{\geq0} \mid \norm{e}<\psi(t)\}.
\]
The approach originated in the work of Ilchmann, Ryan and Sangwin
\cite{Ilchmann2002}, building on the high-gain adaptive-control line
exemplified by Miller and Davison \cite{Miller1991}.  It was subsequently
developed for broad nonlinear system classes of known relative degree
\cite{Ilchmann2005,Ilchmann2007,Berger2021}; a recent overview is given in
\cite{Berger2025}.  A characteristic feature is that the continuous-time
feedback uses the measured tracking error and qualitative information about
the plant class, rather than a numerical plant model.

The continuous-time theory covers input constraints and chemical-reactor
models \cite{Ilchmann2004}, relative-degree-two systems \cite{Hackl2013},
higher-relative-degree designs and funnel pre-compensation
\cite{Berger2018a,Berger2018b}, differential-algebraic and electrical-circuit
models \cite{Berger2014}, and several infinite-dimensional settings which we do not recall here because they are beyond the scope of this present work. These results establish continuous-time funnel feasibility, but they do not by themselves imply that a time-discrete approximation of the singular closed loop remains admissible.

We isolate this numerical question for the fully actuated full-state prototype system
\begin{equation}\label{eq:system}
    \dot y(t)=f(y(t))+g u(t),
\end{equation}
where $y,u:\mathbb R_{\geq0}\to\mathbb R^m$, $g>0$, and
$f:\mathbb R^m\to\mathbb R^m$ is locally Lipschitz. The funnel controller is of the form
\begin{equation}\label{eq:FC}
    u(t)=-\frac{e(t)}{\psi(t)-\norm{e(t)}}.
\end{equation}
The denominator becomes singular as the error approaches the funnel boundary and
prevents boundary contact in continuous time, see
Figure~\ref{Fig:funnel}.  
A general positive definite
high-frequency matrix $G\in\R^{m\times m}$ instead of $gI_m$ in~\eqref{eq:system} requires a separate analysis and is out of the scope of the present paper.

\begin{figure}[t]
\centering
\begin{tikzpicture}[scale=0.43]
\tikzset{>=latex}
  \filldraw[color=gray!25] plot[smooth] coordinates {(0.15,4.7)(0.7,2.9)(4,0.4)(6,1.5)(9.5,0.4)(10,0.333)(10.01,0.331)(10.041,0.3) (10.041,-0.3)(10.01,-0.331)(10,-0.333)(9.5,-0.4)(6,-1.5)(4,-0.4)(0.7,-2.9)(0.15,-4.7)};
  \draw[thick] plot[smooth] coordinates {(0.15,4.7)(0.7,2.9)(4,0.4)(6,1.5)(9.5,0.4)(10,0.333)(10.01,0.331)(10.041,0.3)};
  \draw[thick] plot[smooth] coordinates {(10.041,-0.3)(10.01,-0.331)(10,-0.333)(9.5,-0.4)(6,-1.5)(4,-0.4)(0.7,-2.9)(0.15,-4.7)};
  \draw[thick,fill=lightgray] (0,0) ellipse (0.4 and 5);
  \draw[thick] (0,0) ellipse (0.1 and 0.333);
  \draw[thick,fill=gray!25] (10.041,0) ellipse (0.1 and 0.333);
  \draw[thick] plot[smooth] coordinates {(0,2)(2,1.1)(4,-0.1)(6,-0.7)(9,0.25)(10,0.15)};
  \draw[thick,->] (-2,0)--(12,0) node[right,above]{$t$};
  \draw[thick,dashed](0,0.333)--(10,0.333);
  \draw[thick,dashed](0,-0.333)--(10,-0.333);
  \node [black] at (0,2) {\textbullet};
  \draw[->,thick](4,-3)node[right]{$e(t)$}--(2.5,-0.4);
  \draw[->,thick](3,3)node[right]{$(e(0),0)$}--(0.07,2.07);
  \draw[->,thick](9,3)node[right]{$\psi(t)$}--(7,1.4);
\end{tikzpicture}
\caption{Tracking-error evolution in the performance funnel~$\mathcal F_\psi$ with boundary~$\psi$.}
\label{Fig:funnel}
\end{figure}
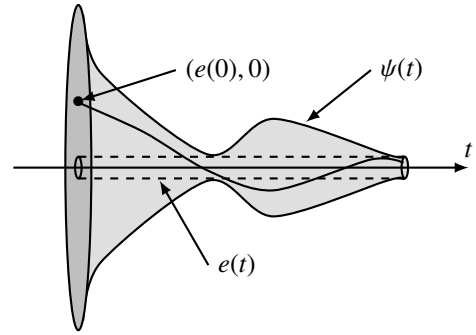

The paper studies numerical integration of the already interconnected
closed-loop model.  Although \eqref{eq:FC} is model-free in its
continuous-time implementation, the Euler recursions analyzed below evaluate
$f$.  They are therefore schemes for simulation or model-based prediction of
the closed-loop dynamics, not direct digital implementations of the
model-free controller.  This scope is also different from sampled-data funnel
control as considered in~\cite{Lanza2024}. In a sampled-data design, sampled measurements generate a held input
for the continuous-time plant and the intersample trajectory must be
controlled.  
Here, by contrast, the
plant and controller are interconnected first and explicit or implicit Euler
is then applied to the resulting differential equation.  The conclusions are
therefore nodal statements and do not imply funnel invariance between grid
points.

The time-varying domain induced by $\mathcal F_\psi$ is the main numerical difficulty.  Classical
consistency and convergence estimates can be invoked only after one knows that
the numerical values stay in a region on which the vector field is defined
and uniformly Lipschitz.  Likewise, standard contractivity results for
implicit Euler under a one-sided Lipschitz condition do not by themselves
establish that the nonlinear update has a root inside the open funnel ball.
The analysis below supplies these missing admissibility estimates directly.
We restrict attention to the two first-order Euler schemes; standard
background on stability and contractivity of one-step methods can be found in
\cite{Hairer1996}.

\paragraph{Contributions}
The paper makes the following contributions.
\begin{enumerate}[label=(C\arabic*),leftmargin=*]
\item For explicit Euler, we compute an upper
      step-size threshold that guarantees a global nodewise funnel margin and prove first-order convergence on finite horizons and, under a~dissipativity condition, uniformly in time.
\item For implicit Euler, an admissible root exists for every step size and 
      every previous value. 
     Furthermore, sufficiently small steps make the admissible root
      unique and yield first-order convergence on finite horizons.
\item If the plant's one-sided Lipschitz expansion is dominated by the funnel 
      damping, the implicit-Euler root is unique for every step size, the
      scheme is contractive, and the first-order error estimate is uniform in
      time.  All error constants are expressed through bounds on the plant,
      reference, funnel, and initial gap.
\end{enumerate}

\paragraph{Organization}
Section~\ref{sec:sysclass} introduces the system class, proves the
continuous-time gap estimate, and collects the monotonicity and regularity
bounds used later.  Section~\ref{sec:timedisc} treats explicit and implicit
Euler.  In Section~\ref{sec:simulation} we conduct a~numerical study.  Section~\ref{sec:conclusion} summarizes the results
and their limitations.

\paragraph{Notation} All norms $\|\cdot\|$ and inner products $\langle\cdot,\cdot\rangle$ in $\R^m$ are Euclidean.  For $R>0$, let $B_R(0)$ and $\overline B_R(0)$ denote the open and closed balls in
$\mathbb R^m$, respectively. For an interval $I\subseteq\mathbb R$, $L^\infty(I,\mathbb R^n)$ is the space of (equivalence classes of) essentially bounded measurable functions $I\to\mathbb R^n$, equipped with the essential-supremum norm $\|\varphi\|_{L^\infty(I)}:=\operatorname*{ess\,sup}_{t\in I}\|\varphi(t)\|$, which we abbreviate as $\|\cdot\|_\infty$.     $W^{1,\infty}(I,\mathbb R^n)$
denotes the space of functions in $L^\infty(I,\mathbb R^n)$ whose
distributional derivative also lies in $L^\infty(I,\mathbb R^n)$.  Every
$\varphi\in W^{1,\infty}(I,\mathbb R^n)$ has a (unique) Lipschitz
representative, which is the one we use throughout.

\section{System class and continuous-time funnel bound}\label{sec:sysclass}

Throughout the paper,
\[
    y_{\rm ref}\in W^{1,\infty}(\mathbb R_{\geq0},\mathbb R^m)
\]
and
\begin{equation*}
    \psi\in\Psi
    :=\left\{\varphi\in W^{1,\infty}(\mathbb R_{\geq0},\mathbb R)
      \;\middle|\;
      \operatorname*{ess\,inf}_{t\geq0}\varphi(t)>0\right\}.
\end{equation*}
We use the Lipschitz representatives of $y_{\rm ref}$ and $\psi$ throughout.
In particular, their values at every grid point are unambiguous and
$\psi(t)>0$ for all $t\geq0$.  

We consider system~\eqref{eq:system}. With $e(t):=y(t)-y_{\rm ref}(t)$, the feedback law \eqref{eq:FC} applied to~\eqref{eq:system} yields the
closed-loop system
\begin{equation}\label{eq:closed-loop-y}
    \dot y(t)=f(y(t))-
    \frac{g\bigl(y(t)-y_{\rm ref}(t)\bigr)}
         {\psi(t)-\|y(t)-y_{\rm ref}(t)\|}.
\end{equation}
Equivalently, the error satisfies, for almost every $t\geq0$,
\begin{equation}\label{eq:closed-loop-e}
    \dot e(t)=F(e(t),t)-\dot y_{\rm ref}(t),
\end{equation}
where
\begin{equation*}
\begin{aligned}
    &F: \mathcal F_\psi\to\R,\ (e,t)\mapsto f\bigl(y_{\rm ref}(t)+e\bigr)-\gamma(e,t),\\
    &\gamma: \mathcal F_\psi\to\R,\ (e,t)\mapsto\frac{g e}{\psi(t)-\|e\|}.
\end{aligned}
\end{equation*}

The following constants will be used throughout:
\begin{equation}\label{eq:data-constants}
\begin{aligned}
    \psi_-&:=\inf_{t\geq0}\psi(t)>0,
    &\psi_+&:=\|\psi\|_{\infty},\\
    Y_0&:=\|y_{\rm ref}\|_{\infty},
    &Y_1&:=\|\dot y_{\rm ref}\|_{\infty},\\
    \Psi_1&:=\|\dot\psi\|_{\infty},
    &K&:=\overline B_{Y_0+\psi_+}(0),\\
    M_f&:=\max_{x\in K}\|f(x)\|,
    &A&:=M_f+Y_1+\Psi_1.
\end{aligned}
\end{equation}
The maximum defining $M_f$ is finite because $K$ is compact and $f$ is
continuous.  For an initial error $e_0$ with
$(e_0,0)\in\mathcal F_\psi$, define
\begin{equation*}
\begin{aligned}
    d_0&:=\psi(0)-\|e_0\|>0,\\
    \eta_*&:=\frac{g\psi_-}{A+g},\\
    \eta_0&:=\min\{d_0,\eta_*\}>0.
\end{aligned}
\end{equation*}
Since $A\geq0$, one has $0<\eta_*\leq\psi_-$.

The following result can in principle be inferred from the available literature, in particular~\cite{Ilchmann2002} and~\cite[Prop.~3.2]{Berger2018a}. In order to be self-contained, we provide a brief proof here.

\begin{theorem}[Continuous-time well-posedness and quantitative funnel margin]
\label{thm:continuous-margin}
For every initial value $y(0)=y_0$ satisfying
$(e_0,0)\in\mathcal F_\psi$ for $e_0:=y_0-y_{\rm ref}(0)$, the closed-loop
system \eqref{eq:closed-loop-y} has a unique global solution.  Its tracking
error satisfies
\begin{equation}\label{eq:continuous-margin}
    \psi(t)-\|e(t)\|\geq\eta_0
    \qquad \forall\, t\geq0.
\end{equation}
In particular, the control input is bounded by
\begin{equation*}
    \|u(t)\|\leq\frac{\psi_+}{\eta_0}
    \qquad \forall\, t\geq0.
\end{equation*}
\end{theorem}

\begin{proof}
The right-hand side of \eqref{eq:closed-loop-y} is continuous in $t$ and
locally Lipschitz in $y$ on the relatively open set
\[
    \Omega:=\{(t,y)\mid (y-y_{\rm ref}(t),t) \in \mathcal F_\psi\}.
\]
Hence, by the standard Carath\'{e}odory local
existence and uniqueness theory \cite[\S10]{Walter1998}, there exists a unique maximal solution $y:[0,\omega)\to\R^m$, $\omega\in (0,\infty]$, of~\eqref{eq:closed-loop-y} and the closure of the graph of~$y$ is not a compact subset of $\Omega$.

Set $d(t):=\psi(t)-\|e(t)\|$ for $t\in [0,\omega)$.  The function $d$ is absolutely
continuous.  Whenever $d(t)<\eta_*$, one has
$\|e(t)\|=\psi(t)-d(t)\geq\psi_--d(t)\ge 0$.  Thus, for almost every such $t\in[0,\omega)$, the
chain rule and \eqref{eq:closed-loop-e} yield
\begin{align*}
    \dot d(t)
    &=\dot\psi(t)-\left\langle\frac{e(t)}{\|e(t)\|},\dot e(t)\right\rangle\\
    &\geq-\Psi_1-M_f-Y_1+g\frac{\|e(t)\|}{d(t)}\\
    &\geq-A+g\frac{\psi_--d(t)}{d(t)}.
\end{align*}
The last expression is strictly positive for $0<d(t)<\eta_*$.  To make the barrier argument explicit, suppose that
$d(t_1)<\eta_0$ for some $t_1<\omega$.  By continuity and $d(0)=d_0\ge \eta_0$, there is 
$t_0\in[0,t_1)$ such that $d(t_0)=\eta_0$ and $d(t)<\eta_0\leq\eta_*$ for
$t\in(t_0,t_1]$.  Hence $\dot d(t)>0$ for almost every $t\in(t_0,t_1)$, and absolute
continuity gives 
\[
  0 >  d(t_1)-d(t_0)=\int_{t_0}^{t_1}\dot d(t)\,\mathrm dt>0, 
\]
a contradiction.
Thus $d(t)\geq\eta_0$ for all $t\in[0,\omega)$.

Consequently, $\|y(t) - y_{\rm ref}(t)\| = \psi(t) - d(t) \le  \psi(t) - \eta_0$ for all $t\in[0,\omega)$. 
If $\omega<\infty$, this implies that the graph of~$y$ is contained in a compact subset of~$\Omega$, a contradiction.
Therefore, $\omega=\infty$.  Finally,
$\|u(t)\|=\|e(t)\|/d(t)\leq\psi_+/\eta_0$ for all $t\ge 0$.
\end{proof}


We next collect two properties needed for the error analysis.  First, the
funnel term is strongly monotone without any differentiability argument.

\begin{lemma}[Strong monotonicity of the funnel term]
\label{lem:gamma-monotone}
For all $(e,t), (z,t)\in \mathcal F_\psi$ we have 
\begin{equation*}
    \langle e-z,\gamma(e,t)-\gamma(z,t)\rangle
    \geq \frac{g}{\psi(t)}\|e-z\|^2
    \geq \frac{g}{\psi_+}\|e-z\|^2.
\end{equation*}
\end{lemma}

\begin{proof} Write $a(\rho):=(\psi(t)-\rho)^{-1}$ for $\rho\in[0,\psi(t))$. Then
\begin{align*}
&~~~~\langle e-z,a(\|e\|)e-a(\|z\|)z\rangle
\\&=a(\|e\|)\|e\|^2+a(\|z\|)\|z\|^2
-(a(\|e\|)+a(\|z\|))\langle e,z\rangle\\
&=\frac{a(\|e\|)+a(\|z\|)}{2}
\bigl(\|e\|^2+\|z\|^2-2\langle e,z\rangle\bigr)\\
&\quad+\frac{a(\|e\|)-a(\|z\|)}{2}(\|e\|^2-\|z\|^2)\\
&=\frac{a(\|e\|)+a(\|z\|)}{2}\|e-z\|^2
+\frac{a(\|e\|)-a(\|z\|)}{2}(\|e\|^2-\|z\|^2).
\end{align*}
Since $a$ is increasing on $[0,\psi(t))$, and since $\rho\mapsto\rho^2$ is increasing on $[0,\infty)$, the two factors $a(\|e\|)-a(\|z\|)$ and
$\|e\|^2-\|z\|^2$ always have the same sign, so that
$$\bigl(a(\|e\|)-a(\|z\|)\bigr)\bigl(\|e\|^2-\|z\|^2\bigr)\geq0.$$
Furthermore, $0\leq\|e\|,\|z\|<\psi(t)$ implies
$a(\|e\|),a(\|z\|)\geq1/\psi(t)$.
Consequently,
$$
 \langle e-z,a(\|e\|)e-a(\|z\|)z\rangle
 \geq \frac1{\psi(t)}\|e-z\|^2.
$$
Recalling that $\gamma(e,t)=g\,a(\|e\|)e$, we obtain
$$
 \langle e-z,\gamma(e,t)-\gamma(z,t)\rangle
 \geq\frac{g}{\psi(t)}\|e-z\|^2.
$$
Finally, $\psi(t)\leq\psi_+$ yields
$$
 \frac{g}{\psi(t)}\|e-z\|^2
 \geq\frac{g}{\psi_+}\|e-z\|^2,
$$
which proves the claim.
\end{proof}

\begin{assumption}[One-sided Lipschitz condition]\label{ass:osl}
For $f$ in~\eqref{eq:system} and $K$ as in~\eqref{eq:data-constants} there exists $\ell_f\in\mathbb R$ such that
\begin{equation}\label{eq:osl}
    \langle x-z,f(x)-f(z)\rangle
    \leq\ell_f\|x-z\|^2
    \qquad \forall x,z\in K,
\end{equation}
and
\begin{equation}\label{eq:alpha}
    \alpha:=\frac{g}{\psi_+}-\ell_f>0.
\end{equation}
\end{assumption}

Conditions \eqref{eq:osl} and \eqref{eq:alpha} are compatible, not
contradictory: \eqref{eq:osl} merely asks for \emph{some} $\ell_f$
bounding the one-sided expansion of $f$ on $K$, which always exists since $f$ is locally Lipschitz (e.g.\ $\ell_f=L_f$ from \eqref{eq:Lf-bound}), while \eqref{eq:alpha} is the basic requirement that this $\ell_f$
lie strictly below $g/\psi_+$.  Since $\psi_+$ enters \eqref{eq:alpha} only on the right-hand side, widening the funnel lowers the required threshold and relaxes \eqref{eq:alpha} monotonically,
without affecting the existence of $\ell_f$ in \eqref{eq:osl}.

\begin{lemma}[Dissipativity of the closed-loop field]
\label{lem:F-dissipative}
Under Assumption~\ref{ass:osl}, for all $(e,t), (z,t)\in \mathcal F_\psi$ we have
\begin{equation}\label{eq:F-dissipative}
    \langle e-z,F(e,t)-F(z,t)\rangle
    \leq-\alpha\|e-z\|^2.
\end{equation}
\end{lemma}

\begin{proof}
The reference shift cancels in the difference.  Combining~\eqref{eq:osl} with Lemma~\ref{lem:gamma-monotone} gives
\begin{align*}
 \langle e-z,F(e,t)-F(z,t)\rangle\leq\left(\ell_f-\frac{g}{\psi_+}\right)\|e-z\|^2
 =-\alpha\|e-z\|^2.
\end{align*}
\end{proof}

\begin{remark}[Interpretation of Assumption~\ref{ass:osl}]
Since $f$ is locally Lipschitz and $K$ is compact, it is
Lipschitz on $K$, that is there exists $L_f>0$ such that
\begin{equation}\label{eq:Lf-bound}
    \forall\, x,y\in K:\ \|f(x)-f(y)\|\le L_f \|x-y\|.
\end{equation}
Hence, it satisfies \eqref{eq:osl} with $\ell_f=L_f$, but the strict dominance condition \eqref{eq:alpha} is an additional incremental stability requirement.  It states that the minimal radial damping $g/\psi_+$ generated by the funnel term dominates the one-sided expansion of the plant.  If $f$ is strictly dissipative with
constant $\beta>0$, then one may take $\ell_f=-\beta$ and
$\alpha=\beta+g/\psi_+$.  Strict dissipativity of the plant is therefore
sufficient but not necessary.
\end{remark}

We call $f$ \emph{strictly dissipative
with constant $\beta>0$} if \eqref{eq:osl} holds with $K=\mathbb R^m$ and
$\ell_f=-\beta$, i.e. $\langle x-z,f(x)-f(z)\rangle\leq-\beta\|x-z\|^2$ for
all $x,z\in\mathbb R^m$.


\begin{lemma}[Lipschitz bound away from the funnel boundary]
\label{lem:F-Lipschitz}
For $\eta>0$, define
\begin{equation}\label{eq:L-F-eta}
    L_F(\eta)
    :=L_f+g\left(\frac{1}{\eta}+\frac{\psi_+}{\eta^2}\right),
\end{equation}
where $L_f>0$ is as in~\eqref{eq:Lf-bound}.
For all $(e,t), (z,t)\in \mathcal F_\psi$ satisfying
$\psi(t)-\|e\|\ge \eta$ and $\psi(t)-\|z\|\geq\eta$ we have
\begin{equation}\label{eq:F-Lipschitz}
    \|F(e,t)-F(z,t)\|\leq L_F(\eta)\|e-z\|.
\end{equation}
\end{lemma}

\begin{proof}
Set $d_e:=\psi(t)-\|e\|$ and $d_z:=\psi(t)-\|z\|$.  Since
$|d_e-d_z|\leq\|e-z\|$ and $\|z\|\leq\psi_+$,
\begin{align*}
 \left\|\frac{e}{d_e}-\frac{z}{d_z}\right\| = \left\|\frac{e-z}{d_e}+\left(\frac{z}{d_e} -\frac{z}{d_z}\right)\right\| 
 &\leq\frac{\|e-z\|}{d_e}
      +\|z\|\frac{|d_e-d_z|}{d_ed_z}\\
 &\leq\left(\frac{1}{\eta}
      +\frac{\psi_+}{\eta^2}\right)\|e-z\|.
\end{align*}
Combining this estimate with~\eqref{eq:Lf-bound} proves~\eqref{eq:F-Lipschitz}.
\end{proof}

Define, using~\eqref{eq:data-constants},
\begin{equation*}
\begin{aligned}
    M_y&:=M_f+\frac{g\psi_+}{\eta_0},\\
    M_e&:=M_y+Y_1,\\
    L_q&:=L_fM_y
       +g\left(\frac{M_e}{\eta_0}
       +\frac{\psi_+(\Psi_1+M_e)}{\eta_0^2}\right).
\end{aligned}
\end{equation*}
These constants are finite and depend only on the problem data and the
initial gap.

\begin{lemma}[Lipschitz regularity along the exact solution]
\label{lem:q-Lipschitz}
Let $y_0\in\R^m$ be such that $(y_0-y_{\rm ref}(0),0)\in\mathcal F_\psi$ and let $y$ be the unique global solution of~\eqref{eq:closed-loop-y} with $y(0)=y_0$ according to Theorem~\ref{thm:continuous-margin}. Set $e(t):=y(t)-y_{\rm ref}(t)$ and
$q(t):=F(e(t),t)$.  Then
\begin{equation}\label{eq:q-Lipschitz}
    \|q(t)-q(s)\|\leq L_q|t-s|
    \qquad \forall s,t\geq0.
\end{equation}
Furthermore, for a time step $h>0$, we have the following estimates for all $t\ge 0$
\begin{align}
    \left\|\int_{t}^{t+h}[q(s)-q(t)]\,\mathrm ds\,\right\|
    &\leq\frac{L_q}{2}h^2,\label{eq:forward-defect}\\
    \left\|\int_{t}^{t+h}[q(s)-q(t+h)]\,\mathrm ds\,\right\|
    &\leq\frac{L_q}{2}h^2.\label{eq:backward-defect}
\end{align}
\end{lemma}

\begin{proof} 
From~\eqref{eq:system} we find that
\[
    \|\dot y(t)\| \le \|f(y(t))\| + g \|u(t)\|
\]
and since $y(t)\in K$ for all $t\ge 0$ and $\|u(t)\|\le \psi_+/\eta_0$ by
Theorem~\ref{thm:continuous-margin}, it follows that
$\|\dot y(t)\|\leq M_y$ and hence $\|\dot e(t)\|\leq M_e$  for almost all $t\ge 0$.  Hence
\[
    \|y(t)-y(s)\|\leq M_y|t-s|,
    \qquad
    \|e(t)-e(s)\|\leq M_e|t-s|.
\]
With $d(t):=\psi(t)-\|e(t)\|\geq\eta_0$ by~\eqref{eq:continuous-margin}, one also has
\[
    |d(t)-d(s)|\leq(\Psi_1+M_e)|t-s|.
\]
Therefore,
\begin{align*}
\left\|\frac{e(t)}{d(t)}-\frac{e(s)}{d(s)}\right\|
&\leq \left\|\frac{e(t)-e(s)}{d(t)}+e(s)\left(\frac{1}{d(t)}-\frac{1}{d(s)}\right)\right\|\\
&\leq
\left(\frac{M_e}{\eta_0}
 +\frac{\psi_+(\Psi_1+M_e)}{\eta_0^2}\right)|t-s|.
\end{align*}
Together with $\|f(y(t))-f(y(s))\|\leq L_fM_y|t-s|$ by~\eqref{eq:Lf-bound},
this proves \eqref{eq:q-Lipschitz}.  Integrating over
$[t,t+h]$ yields \eqref{eq:forward-defect} and
\eqref{eq:backward-defect}.
\end{proof}

\section{Euler time discretization}\label{sec:timedisc}
The explicit and implicit Euler schemes are obtained by discretizing the closed-loop system~\eqref{eq:closed-loop-y} and then defining
$e_k:=y_k-y_{\rm ref}(t_k)$, where $t_k:=kh$ for $k\in\N_0$ and $y_0\in\R^m$ is such that $(e_0,0)\in\mathcal F_\psi$. Introduce the exact reference
increment
\begin{equation*}
    \Delta y_{{\rm ref},k}
    :=y_{\rm ref}(t_{k+1})-y_{\rm ref}(t_k),
\end{equation*}
with which we avoid noncanonical pointwise values of the weak derivative
$\dot y_{\rm ref}$.

\subsection{Explicit Euler}\label{subsec:explicit}
The explicit Euler scheme is given by
\begin{equation}\label{eq:explicit-Euler}
    e_{k+1}=e_k+hF(e_k,t_k)-\Delta y_{{\rm ref},k}.
\end{equation}
The first result is a direct, global-in-time nodewise funnel estimate.  It does
not rely on convergence to the exact solution.

\begin{theorem}[Global nodewise funnel preservation for explicit Euler]
\label{thm:explicit-margin}
Fix any desired discrete margin
\begin{equation}\label{eq:eta-E-choice}
    0<\eta_{\rm E}<\eta_0
\end{equation}
and define, using~\eqref{eq:data-constants},
\begin{equation}\label{eq:h-E}
    \overline h_{\rm E}:=
    \begin{cases}
    \displaystyle
    \min\left\{\frac{\eta_{\rm E}}{g},
    \frac{\eta_*-\eta_{\rm E}}{A}\right\},&A>0,\\[2ex]
    \displaystyle\frac{\eta_{\rm E}}{g},&A=0.
    \end{cases}
\end{equation}
Then $\overline h_{\rm E}>0$, and for every fixed step size
$0<h\leq\overline h_{\rm E}$ the recursion \eqref{eq:explicit-Euler} is
well defined for all $k\in\mathbb N_0$ and satisfies
\begin{equation}\label{eq:explicit-margin}
    \psi(t_k)-\|e_k\|\geq\eta_{\rm E}
    \qquad \forall\, k\in\mathbb N_0.
\end{equation}
\end{theorem}

\begin{proof} Let $k\in\N_0$ be such that $(e_k,t_k)\in\mathcal F_\psi$ and $d_k:=\psi(t_k)-\|e_k\|\geq\eta_{\rm E}$, which is clear for $k=0$. We show by induction that this also holds for $e_{k+1}$. First,~\eqref{eq:h-E} gives $0\le h \le \overline h_{\rm E} \le \frac{\eta_{\rm E}}{g} \le \frac{d_k}{g}$, thus $0\leq1-hg/d_k\leq1$.  Since
\[
 e_{k+1}=\left(1-\frac{hg}{d_k}\right)e_k
          +hf(y_{\rm ref}(t_k)+e_k)-\Delta y_{{\rm ref},k},
\]
the triangle inequality, $|1-hg/d_k| = 1-hg/d_k$, $y_{\rm ref}(t_k)+e_k\in K$ and
$\|\Delta y_{{\rm ref},k}\|\leq hY_1$ imply
\[
    \|e_{k+1}\|\leq \|e_k\|+h(M_f+Y_1)-hg\frac{\|e_k\|}{d_k}.
\]
Using $\psi(t_{k+1})-\psi(t_k)\geq-h\Psi_1$ and recalling~\eqref{eq:data-constants}therefore yields
\begin{equation}\label{eq:explicit-gap-increment}
    d_{k+1}-d_k
    \geq h\left[-A+g\frac{\|e_k\|}{d_k}\right]
    \geq h\left[-A+g\frac{\psi_--d_k}{d_k}\right].
\end{equation}
If $d_k\leq\eta_*$, the right-hand side is nonnegative by the definition of $\eta_*$.  Hence $d_{k+1}\geq d_k\geq\eta_{\rm E}$.  If $d_k>\eta_*$, then the first inequality of~\eqref{eq:explicit-gap-increment} gives $d_{k+1}\geq d_k-hA$.  For $A>0$, the second restriction in \eqref{eq:h-E} implies
$d_{k+1}>\eta_*-hA\geq\eta_{\rm E}$; for $A=0$ we have $d_{k+1}\ge d_k\ge \eta_{\rm E}$. In particular, $(e_{k+1},t_{k+1})\in\mathcal F_\psi$ and the induction is complete.
\end{proof}

\begin{remark}
The simple choice $\eta_{\rm E}=\eta_0/2$ makes
\eqref{eq:h-E} completely automatic.  The two terms in the threshold reflect
different mechanisms: $h\leq\eta_{\rm E}/g$ prevents a sign reversal of the
radial damping factor, whereas the second term controls the maximal one-step
loss of gap when the current error is far from the singular boundary.
\end{remark}

For the error estimates, set $L_{\rm E}:=L_F(\eta_{\rm E})$ for $L_F(\cdot)$ defined in~\eqref{eq:L-F-eta}. 
Lemma~\ref{lem:F-Lipschitz} then gives
\begin{equation}\label{eq:F-Lipschitz-E}
    \|F(e,t)-F(z,t)\|\leq L_{\rm E}\|e-z\|
\end{equation}
for all $(e,t), (z,t)\in \mathcal F_\psi$ with
$\psi(t)-\|e\|\ge \eta_{\rm E}$ and $\psi(t)-\|z\|\geq\eta_{\rm E}$.

\begin{theorem}[Finite-horizon first-order convergence of explicit Euler]
\label{thm:explicit-finite-error}
Let $y_0\in\R^m$ be such that $(y_0-y_{\rm ref}(0),0)\in\mathcal F_\psi$ and let $y$ be the unique global solution of~\eqref{eq:closed-loop-y} with $y(0)=y_0$ according to Theorem~\ref{thm:continuous-margin}. Set $e(t):=y(t)-y_{\rm ref}(t)$ and consider, under the assumptions of Theorem~\ref{thm:explicit-margin} and with $0<h\leq\overline h_{\rm E}$, the sequence $(e_k)$ obtained from~\eqref{eq:explicit-Euler} with $e_0=e(0)$. Then, for every
$T>0$,
\begin{equation}\label{eq:explicit-finite-error}
    \|e_k-e(t_k)\|
    \leq C_{\rm E}(T)h
    \qquad (t_k = kh \leq T),
\end{equation}
where
\begin{equation*}
    C_{\rm E}(T)
    :=\frac{L_q}{2}\frac{e^{L_{\rm E}T}-1}{L_{\rm E}}.
\end{equation*}
\end{theorem}
\begin{proof}
The integral form of \eqref{eq:closed-loop-e} gives
\begin{equation}\label{eq:exact-forward-step}
e(t_{k+1})
=e(t_k)+hF(e(t_k),t_k)
-\Delta y_{{\rm ref},k}+\rho_{k+1},
\end{equation}
where
\begin{equation}\label{eq:rho-def}
\rho_{k+1}:=\int_{t_k}^{t_{k+1}}\bigl[q(s)-q(t_k)\bigr]\,\mathrm ds,
    \qquad q(s):=F(e(s),s),
\end{equation}
so that by \eqref{eq:forward-defect},  $\|\rho_{k+1}\|\leq \frac{L_q}{2}h^2.$ Set
$E_k:=e(t_k)-e_k.$
Subtracting \eqref{eq:explicit-Euler} from
\eqref{eq:exact-forward-step} yields
\begin{equation}\label{eq:error-identity}
    E_{k+1}
    =E_k
    +h\bigl(F(e(t_k),t_k)-F(e_k,t_k)\bigr)
    +\rho_{k+1}.
\end{equation}
The continuous margin \eqref{eq:continuous-margin} gives
$$
    \psi(t_k)-\|e(t_k)\|
    \geq \eta_0>\eta_{\rm E},
$$
whereas the discrete margin \eqref{eq:explicit-margin} gives
$$
    \psi(t_k)-\|e_k\|
    \geq \eta_{\rm E}.
$$
Hence \eqref{eq:F-Lipschitz-E} applies to $e(t_k)$ and $e_k$, and
therefore
\begin{equation}\label{eq:explicit-error-recursion}
\|E_{k+1}\|
\leq
(1+hL_{\rm E})\|E_k\|
+\frac{L_q}{2}h^2.
\end{equation}
Since $E_0=0$, iterating \eqref{eq:explicit-error-recursion} gives
\[
\|E_k\|
\leq
\frac{L_q}{2}h^2
\sum_{j=0}^{k-1}(1+hL_{\rm E})^j\
=
\frac{L_q}{2L_{\rm E}}
\bigl((1+hL_{\rm E})^k-1\bigr)h.
\]
Using $1+x\leq e^x$ for $x\ge 0$ and $t_k=kh$, we obtain
$$
    (1+hL_{\rm E})^k
    \leq e^{khL_{\rm E}}
    =e^{L_{\rm E}t_k}.
$$
Consequently, 
$$
    \|E_k\|
    \leq
    \frac{L_q}{2L_{\rm E}}
    \left(e^{L_{\rm E}t_k}-1\right)h.
$$
For $t_k\leq T$, it follows that
$$
    \|e(t_k)-e_k\|
    =\|E_k\|
    \leq
    \frac{L_q}{2L_{\rm E}}
    \left(e^{L_{\rm E}T}-1\right)h
    =C_{\rm E}(T)h,
$$
whence~\eqref{eq:explicit-finite-error}.
\end{proof}

Under the additional one-sided Lipschitz condition (Assumption~\ref{ass:osl}), the same direct estimate
also yields a uniform-in-time error bound for sufficiently small steps.

\begin{theorem}[Uniform-in-time explicit-Euler error]
\label{thm:explicit-uniform-error}
Let the assumptions of Theorem~\ref{thm:explicit-finite-error} and, additionally, Assumption~\ref{ass:osl} hold and fix $\eta_{\rm E}$ as in~\eqref{eq:eta-E-choice}.  If
\begin{equation}\label{eq:h-explicit-uniform}
    0<h\leq\min\left\{\overline h_{\rm E},
                       \frac{\alpha}{L_{\rm E}^2}\right\},
\end{equation}
then we have
\begin{equation*}
    \|e_k-e(t_k)\|\leq\frac{L_q}{\alpha}h
    \qquad \forall\, k\in\mathbb N_0.
\end{equation*}
\end{theorem}

\begin{proof}
For all $(e,t), (z,t)\in \mathcal F_\psi$ with
$\psi(t)-\|e\|\ge \eta_{\rm E}$ and $\psi(t)-\|z\|\geq\eta_{\rm E}$,
\eqref{eq:F-dissipative} and \eqref{eq:F-Lipschitz-E} imply
\begin{align*}
 \|(e-z)+h(F(e,t)-F(z,t))\|^2\leq
 \bigl(1-2\alpha h+L_{\rm E}^2h^2\bigr)\|e-z\|^2.
\end{align*}
Set
$q_{\rm E}:=(1-2\alpha h+L_{\rm E}^2h^2)^{1/2}$.  
Then~\eqref{eq:h-explicit-uniform} gives $0\leq\alpha h\leq\alpha^2/L_{\rm E}^2\leq1$ and
$q_{\rm E}^2\leq1-\alpha h$.  Since
$1-\sqrt{1-x}\geq x/2$ for $x\in[0,1]$, it follows for $x=\alpha h$ that
\[
    1-q_{\rm E}\geq 1-\sqrt{1-\alpha h} \ge \alpha h/2.
\]
The error identity~\eqref{eq:error-identity} together with $\|\rho_{k+1}\|\leq \frac{L_q}{2}h^2$ by \eqref{eq:forward-defect} then gives
\[
    \|E_{k+1}\|\leq q_{\rm E}\|E_k\|+\frac{L_q}{2}h^2.
\]
Iteration and $E_0=0$ yield
\[
    \|E_k\|\leq\frac{L_qh^2}{2(1-q_{\rm E})}
    \leq\frac{L_q}{\alpha}h.
\]
\end{proof}

\subsection{Implicit Euler}\label{subsec:implicit}
The implicit Euler scheme is given by
\begin{equation}\label{eq:implicit-Euler}
    e_{k+1}=e_k+hF(e_{k+1},t_{k+1})
    -\Delta y_{{\rm ref},k}.
\end{equation}
An \emph{admissible solution} of this equation is, by definition, a sequence $(e_k)$
satisfying $\|e_{k+1}\|<\psi(t_{k+1})$ (i.e., $(e_{k+1},t_{k+1})\in \mathcal F_\psi$) for all $k\in\N_0$. 

\begin{proposition}[Admissible solvability and discrete funnel margin]
\label{prop:implicit-solvability}
For every $h>0$ and every $e_k\in\mathbb R^m$, equation
\eqref{eq:implicit-Euler} has at least one admissible solution $e_{k+1}\in\R^m$.  If $\psi(0) - \|e_0\|\ge \eta_0$, then every sequence $(e_k)$ of admissible solutions satisfies
\begin{equation}\label{eq:implicit-margin}
    \psi(t_k)-\|e_k\|\geq\eta_0
    \qquad \forall\, k\in\mathbb N_0.
\end{equation}
If Assumption~\ref{ass:osl} holds, the admissible solution $e_{k+1}$ of~\eqref{eq:implicit-Euler} at each step is unique for every $h>0$.
\end{proposition}

\begin{proof}
Fix $k\in\N_0$, abbreviate $b_k:=e_k-\Delta y_{{\rm ref},k}$, and define
\begin{multline*}
    G_{k+1}:B_{\psi(t_{k+1})}(0)\to \R^m,\\ \xi \mapsto\xi-b_k-hf(y_{\rm ref}(t_{k+1})+\xi)
      +h\frac{g\xi}{\psi(t_{k+1})-\|\xi\|}.
\end{multline*}
For $\|\xi\|<\psi(t_{k+1})$ we have $y_{\rm ref}(t_{k+1})+\xi \in K$ for $K$ as in~\eqref{eq:data-constants}, so $\norm{f(y_{\rm ref}(t_{k+1})+\xi)}\le M_f$ and
\[
 \langle G_{k+1}(\xi),\xi\rangle
 \geq \|\xi\|^2-(\|b_k\|+hM_f)\|\xi\|+h\frac{g\|\xi\|^2}{\psi(t_{k+1})-\|\xi\|}.
\]
The last term tends to $+\infty$ as $\|\xi\|\uparrow \psi(t_{k+1})$.  Hence there exists
$R\in(0,\psi(t_{k+1}))$ such that
$\langle G_{k+1}(\xi),\xi\rangle>0$ on $\|\xi\|=R$.  If $G_{k+1}$ had no
zero in $\overline B_R(0)$, then the continuous map
\[
    \Phi:\overline B_R(0)\to\overline B_R(0),\ \xi\mapsto -R\frac{G_{k+1}(\xi)}{\|G_{k+1}(\xi)\|}
\]
would have a 
fixed point $\xi^*\in \overline B_R(0)$ by Brouwer's fixed-point theorem. Since
$\|\Phi(\xi)\|=R$ for all $\xi\in \overline B_R(0)$, this fixed point lies on the sphere and satisfies
$G_{k+1}(\xi^*)=-\lambda\xi^*$ for $\lambda:= \|G_{k+1}(\xi^*)\|/R>0$, thus
\[
    \langle G_{k+1}(\xi^*),\xi^*\rangle = -\lambda \norm{\xi^*}^2 < 0,
\]
contradicting $\langle G_{k+1}(\xi),\xi\rangle>0$ on $\|\xi\|=R$.  Thus there exists $e_{k+1}\in B_{\psi(t_{k+1})}(0)$ with $G_{k+1}(e_{k+1})=0$, which is a solution of~\eqref{eq:implicit-Euler}.

For the margin, set $d_k:=\psi(t_k)-\|e_k\|$.  If $e_{k+1}=0$, then
$d_{k+1}=\psi(t_{k+1})\geq\psi_-\geq\eta_0$.  Otherwise, take the inner
product of~\eqref{eq:implicit-Euler} with
$e_{k+1}/\norm{e_{k+1}}$.  This gives
\[
    \norm{e_{k+1}}\leq \|e_k\|+h(M_f+Y_1)
    -hg\frac{\norm{e_{k+1}}}{d_{k+1}},
\]
and therefore, using~\eqref{eq:data-constants},
\begin{align}
    d_{k+1}-d_k &\ge \psi(t_{k+1}) - \psi(t_k) - \big(\norm{e_{k+1}} - \norm{e_k}\big)\notag \\
    &\ge -h \Psi_1 - h(M_f+Y_1) + hg\frac{\norm{e_{k+1}}}{d_{k+1}} \notag \\
    &\geq h\left(-A+g\frac{\psi_--d_{k+1}}{d_{k+1}}\right). \label{eq:implicit-gap-increment}
\end{align}

If $d_k\geq\eta_0$ but $d_{k+1}<\eta_0\leq\eta_*$, the right-hand side of~\eqref{eq:implicit-gap-increment} is positive by the definition of $\eta^*$, so that
$d_{k+1}>d_k\geq\eta_0$, a contradiction.  Induction using $d_0\ge \eta_0$ proves~\eqref{eq:implicit-margin}.

Finally, suppose that $e_{k+1}$ and $z_{k+1}$ are two admissible solutions of~\eqref{eq:implicit-Euler}
for the same previous value~$e_k$ and that Assumption~\ref{ass:osl} holds.  Subtracting
the equations and taking the inner product with
$\delta:=e_{k+1}-z_{k+1}$ gives
\[
    \|\delta\|^2
    =h\langle\delta,F(e_{k+1},t_{k+1})
                         -F(z_{k+1},t_{k+1})\rangle
    \leq-h\alpha\|\delta\|^2
\]
by Lemma~\ref{lem:F-dissipative} since $\|e_{k+1}\|<\psi(t_{k+1})$ and $\|z_{k+1}\|<\psi(t_{k+1})$.
Thus $(1+h\alpha)\|\delta\|^2\leq0$, yielding $\delta=0$.
\end{proof}

\begin{remark}[Root selection]
Proposition~\ref{prop:implicit-solvability} concerns roots in the open funnel
ball.  Without Assumption~\ref{ass:osl}, the implicit relation can genuinely
be set-valued. For example, let $m=1$, $\psi\equiv1$,
$y_{\rm ref}\equiv0$, $g=1$, $f(y)=ay$, and $e_k=0$.  If
$a>1+1/h$, then the implicit equation has the three admissible roots
\[
    0,\qquad
    \pm\left(1-\frac{1}{a-1/h}\right).
\]
Thus existence and the nodal margin alone do not select a numerical branch.
The next theorem shows that sufficiently small steps make the admissible root
unique along the trajectory initialized at $e(0)$.  If the rational feedback
is artificially extended beyond the funnel boundary, additional exterior
algebraic roots may also occur.  A numerical implementation should therefore
solve the constrained problem or use a parameterization that enforces
admissibility.
\end{remark}

For the error analysis, set 
    $L_{\rm I}:=L_F(\eta_0)$. 
The exact solution of~\eqref{eq:closed-loop-e} satisfies the backward-step identity
\begin{equation}\label{eq:exact-backward-step}
    e(t_{k+1})=e(t_k)+hF(e(t_{k+1}),t_{k+1})
    -\Delta y_{{\rm ref},k}+\tau_{k+1},
\end{equation}
where
\begin{equation}\label{eq:tau-def}
    \tau_{k+1}:=\int_{t_k}^{t_{k+1}}\bigl[q(s)-q(t_{k+1})\bigr]\,\mathrm ds,
\end{equation}
with $q$ as in~\eqref{eq:rho-def}, so that \eqref{eq:backward-defect} gives
\begin{equation}\label{eq:backward-defect-tau}
    \|\tau_{k+1}\|\leq\frac{L_q}{2}h^2.
\end{equation}

\begin{theorem}[Finite-horizon first-order convergence of implicit Euler]
\label{thm:implicit-finite-error}
Let $y_0\in\R^m$ be such that $(y_0-y_{\rm ref}(0),0)\in\mathcal F_\psi$ and let $y$ be the unique global solution of~\eqref{eq:closed-loop-y} with $y(0)=y_0$ according to Theorem~\ref{thm:continuous-margin}. Set $e(t):=y(t)-y_{\rm ref}(t)$ and consider, for $e_0\in\R^m$ with $\psi(0) - \|e_0\|\ge \eta_0$ and with
\begin{equation}\label{eq:h-implicit-finite}
    0<h\leq\frac{1}{2L_{\rm I}},
\end{equation}
\textit{some} admissible solution $(e_k)$ obtained from~\eqref{eq:implicit-Euler} with $e_0=e(0)$. Then $e_{k+1}$ is the unique solution of \eqref{eq:implicit-Euler} for every $k\in\N_0$. Furthermore, for every $T>0$,
\begin{equation}\label{eq:implicit-finite-error}
    \|e_k-e(t_k)\|\leq C_{\rm I}(T)h
    \qquad (t_k\leq T),
\end{equation}
where
\begin{equation*}
    C_{\rm I}(T)
    :=\frac{L_q}{2L_{\rm I}}
      \left(e^{2L_{\rm I}T}-1\right).
\end{equation*}
\end{theorem}

\begin{proof} We prove the assertion on uniqueness by induction on $k$. For $k=0$ the value of $e_0$ is unique by assumption. Assume that for some $k\in\N_0$, $e_{k+1}$ and $z_{k+1}$ are two solutions of~\eqref{eq:implicit-Euler}. 
By Proposition~\ref{prop:implicit-solvability} we have that~\eqref{eq:implicit-margin} holds for $e_{k+1}$ and $z_{k+1}$.
Then Lemma~\ref{lem:F-Lipschitz} together with $L_{\rm I}=L_F(\eta_0)$ gives
\[
    \|e_{k+1}-z_{k+1}\| = h\norm{F(e_{k+1},t_{k+1}) - F(z_{k+1},t_{k+1})}
    \leq hL_{\rm I}\|e_{k+1}-z_{k+1}\|.
\]
Condition \eqref{eq:h-implicit-finite} therefore implies $e_{k+1}=z_{k+1}$, which completes the induction.

Now let $E_k:=e_k-e(t_k)$.  Subtracting
\eqref{eq:exact-backward-step} from \eqref{eq:implicit-Euler} yields
\begin{equation}\label{eq:error-implicit}
    E_{k+1}
    =E_k
    +h\bigl(F(e_{k+1},t_{k+1})-F(e(t_{k+1}),t_k)\bigr)
    -\tau_{k+1}.
\end{equation}
The continuous margin \eqref{eq:continuous-margin} gives
$$
    \psi(t_{k+1})-\|e(t_{k+1})\|
    \geq \eta_0,
$$
whereas the discrete margin \eqref{eq:implicit-margin} gives
$$
    \psi(t_{k+1})-\|e_{k+1}\|
    \geq \eta_0.
$$
Hence Lemma~\ref{lem:F-Lipschitz} applies to $e(t_{k+1})$ and $e_{k+1}$, and
therefore
\[
    \norm{F(e_{k+1},t_{k+1}) - F(e(t_{k+1}),t_{k+1})}
    \leq L_{\rm I}\|E_{k+1}\|.
\]
Together with~\eqref{eq:backward-defect-tau} this yields
\[
    (1-hL_{\rm I})\|E_{k+1}\|
    \leq\|E_k\|+\frac{L_q}{2}h^2.
\]
Set $q_h:=(1-hL_{\rm I})^{-1}$.  Condition~\eqref{eq:h-implicit-finite} gives $0\le hL_{\rm I}\le1/2$, so $1-hL_{\rm I}\ge1/2>0$ and $q_h$ is well defined and positive.  To see $q_h\leq e^{2hL_{\rm I}}$, set $x:=hL_{\rm I}\in[0,1/2]$; it suffices to show $g(x):=e^{2x}(1-x)\geq1$ on $[0,1/2]$, since this is equivalent to $1-x\geq e^{-2x}=1/e^{2x}$, i.e.\ $q_h=(1-x)^{-1}\leq e^{2x}$.  We have $g(0)=1$ and
\[
    g'(x)=2e^{2x}(1-x)-e^{2x}=e^{2x}(1-2x)\geq0
    \qquad\text{for }x\in[0,1/2],
\]
so $g$ is nondecreasing on $[0,1/2]$ and hence $g(x)\geq g(0)=1$ there, which is the claimed inequality.  Iteration from
$E_0=0$ gives
\begin{align*}
    \|E_k\|
    &\leq\frac{L_q}{2}h^2
           \sum_{j=1}^{k}q_h^j=\frac{L_q}{2}h^2\,q_h\,\frac{q_h^k-1}{q_h-1}\\
   &=\frac{L_q}{2}h\,\frac{q_h}{hL_{\rm I}/(1-hL_{\rm I})}\bigl(q_h^k-1\bigr)
    \qquad\text{(since }q_h-1=\tfrac{hL_{\rm I}}{1-hL_{\rm I}}\text{)}\\
    &=\frac{L_q}{2L_{\rm I}}h\,(1-hL_{\rm I})q_h\bigl(q_h^k-1\bigr)
    =\frac{L_q}{2L_{\rm I}}h\,(q_h^k-1)\\
    &\leq\frac{L_q}{2L_{\rm I}}
           \left(e^{2L_{\rm I}t_k}-1\right)h,
\end{align*}
whence \eqref{eq:implicit-finite-error}.
\end{proof}

Under
Assumption~\ref{ass:osl}, the implicit-Euler flow is contractive in the following sense, resulting in
two admissible trajectories with different initial gaps to merge geometrically,
independent of the step size $h>0$.

\begin{lemma}[Contractivity of implicit Euler]
\label{lem:implicit-contractivity}
Let Assumption~\ref{ass:osl} hold and let $(e_k)$ and $(z_k)$ be two admissible solutions of~\eqref{eq:implicit-Euler}.  Then
\begin{equation}\label{eq:implicit-contractivity}
    \|e_{k+1}-z_{k+1}\|
    \leq\frac{1}{1+h\alpha}\|e_k-z_k\|,
\end{equation}
and hence
\[
    \|e_k-z_k\|\leq(1+h\alpha)^{-k}\|e_0-z_0\|.
\]
\end{lemma}
\begin{proof}
Subtracting the two equations of form~\eqref{eq:implicit-Euler} gives
\[
 \delta_{k+1}=\delta_k+h\bigl(F(e_{k+1},t_{k+1})
                                  -F(z_{k+1},t_{k+1})\bigr),
 \qquad \delta_j:=e_j-z_j.
\]
Taking the inner product with $\delta_{k+1}$ and using
Lemma~\ref{lem:F-dissipative} yields
\[
 (1+h\alpha)\|\delta_{k+1}\|^2
 \leq\|\delta_k\|\,\|\delta_{k+1}\|.
\]
Division by $\|\delta_{k+1}\|$ when it is nonzero proves
\eqref{eq:implicit-contractivity}; the zero case is immediate.
\end{proof}

\begin{theorem}[Uniform-in-time implicit-Euler error]
\label{thm:implicit-uniform-error}
Let Assumption~\ref{ass:osl} hold and let $y_0\in\R^m$ be such that $(y_0-y_{\rm ref}(0),0)\in\mathcal F_\psi$ and let $y$ be the unique global solution of~\eqref{eq:closed-loop-y} with $y(0)=y_0$ according to Theorem~\ref{thm:continuous-margin}. Set $e(t):=y(t)-y_{\rm ref}(t)$ and let $(e_k)$ be the unique admissible solution of~\eqref{eq:implicit-Euler} with $e_0=e(0)$ according to Proposition~\ref{prop:implicit-solvability}. Then, for every $h>0$,
\begin{equation}\label{eq:implicit-uniform-error}
    \|e_k-e(t_k)\|\leq\frac{L_q}{2\alpha}h
    \qquad \forall\, k\in\mathbb N_0.
\end{equation}
Consequently, a prescribed nodal accuracy $\delta>0$ is guaranteed by
\begin{equation}\label{eq:implicit-accuracy-step}
    0<h\leq\frac{2\alpha\delta}{L_q}.
\end{equation}
\end{theorem}

\begin{proof}
Put $E_k:=e_k-e(t_k)$, then~\eqref{eq:error-implicit} holds. 
Taking the inner product with $E_{k+1}$ yields
\begin{align*}
\|E_{k+1}\|^2
&=
\langle E_k,E_{k+1}\rangle
+h\left\langle
F(e_{k+1},t_{k+1})
-F(e(t_{k+1}),t_{k+1}),
E_{k+1}
\right\rangle\
\\&\quad -\langle\tau_{k+1},E_{k+1}\rangle.
\end{align*}
By~\eqref{eq:F-dissipative} we have
$$
    \left\langle
        F(e_{k+1},t_{k+1})
        -F(e(t_{k+1}),t_{k+1}),
        E_{k+1}
    \right\rangle
    \leq
    -\alpha\|E_{k+1}\|^2.
$$
and~\eqref{eq:backward-defect-tau} gives
$$
    -\langle\tau_{k+1},E_{k+1}\rangle
    \leq
    \|\tau_{k+1}\|\,\|E_{k+1}\|
    \leq
    \frac{L_q}{2}h^2\|E_{k+1}\|,
$$
thus
$$
    (1+h\alpha)\|E_{k+1}\|^2
    \leq
    \left(
        \|E_k\|
        +\frac{L_q}{2}h^2
    \right)
    \|E_{k+1}\|.
$$
If $E_{k+1}=0$, the desired estimate is immediate. Otherwise,
division by $\|E_{k+1}\|$ gives
\begin{equation}\label{eq:implicit-error-recursion}
(1+h\alpha)\|E_{k+1}\|
\leq
\|E_k\|
+\frac{L_q}{2}h^2.
\end{equation}
Define $    q_\alpha:=\frac{1}{1+h\alpha}.$
 Since $h>0$ and $\alpha>0$, we have $0<q_\alpha<1$. Thus
\eqref{eq:implicit-error-recursion} is equivalent to
\begin{equation}\label{eq:implicit-error-recursion-q}
\|E_{k+1}\|
\leq
q_\alpha\|E_k\|
+q_\alpha\frac{L_q}{2}h^2.
\end{equation}
Iterating this estimate and using $E_0=0$ gives, similar to the proof of Theorem~\ref{thm:implicit-finite-error},
\begin{align*}    
    \|E_k\|
    &\leq
    \frac{L_q}{2}h^2
    \sum_{j=1}^{k}q_\alpha^j
    =
    q_\alpha\frac{L_q}{2}h^2
    \sum_{j=0}^{k-1}q_\alpha^j\\
    &\leq 
   q_\alpha\frac{L_q}{2}h^2
    \frac{1}{1-q_\alpha}
    \le \frac{L_q}{2}h^2\frac{1}{h\alpha},
\end{align*}
whence \eqref{eq:implicit-uniform-error}. Finally, \eqref{eq:implicit-accuracy-step} with \eqref{eq:implicit-uniform-error} guarantees an accuracy of~$\delta$, as $\|e_k-e(t_k)\|\le \delta$.
\end{proof}

\begin{remark}
The implicit scheme preserves a quantitative funnel margin for every
$h>0$; condition \eqref{eq:implicit-accuracy-step} concerns accuracy only.
Under Assumption~\ref{ass:osl}, explicit Euler also has a uniform error bound,
but only under the additional stability and invariance restriction
\eqref{eq:h-explicit-uniform}.  This is the essential qualitative distinction
between the two Euler schemes in the present analysis.
\end{remark}

\section{Numerical study}\label{sec:simulation}
This section illustrates the preceding results by simulating the
funnel-controlled tracking error under explicit and implicit Euler
discretization for varying step sizes. We consider system~\eqref{eq:system} with $f(x)=-0.3x^3$ and $g=1$, in
closed loop with the funnel controller~\eqref{eq:FC} for the prescribed
performance funnel $\psi(t)=15e^{-t}+0.3$ and the sinusoidal reference
trajectory $y_{\rm ref}(t)=\sin(t)$.  The initial tracking error is
$e_0=0.5$ in both cases, and the horizon is $[0,T]$ with $T=10$, which is
long enough for $\psi$ to have essentially reached its steady-state value $0.3$.  All simulations are implemented in Python, using the NumPy,
Matplotlib, and SciPy libraries.

For both the schemes we use the step sizes $h\in\{0.05,0.1,0.3,0.5\}$. 
The explicit Euler scheme~\eqref{eq:explicit-Euler} is iterated directly, and at each step we check that the update stays inside the funnel, i.e.\ that the barrier
denominator $\psi(t_{k+1})-\|e_{k+1}\|$ remains positive.  As the step size increases, the discretization error accumulates until the numerical trajectory overshoots the (shrinking) funnel boundary, at which point the
simulation is terminated.  Figure~\ref{fig:explicit} shows this behavior. \\
For implicit Euler, we solve for the admissible root of \eqref{eq:implicit-Euler}
numerically, searching only inside the open funnel ball;
Proposition~\ref{prop:implicit-solvability} guarantees that such a root exists.  As Figure~\ref{fig:implicit} illustrates, the tracking error indeed never leaves the prescribed performance funnel.
\begin{figure}
    \centering
    \includegraphics[width=0.9\linewidth]{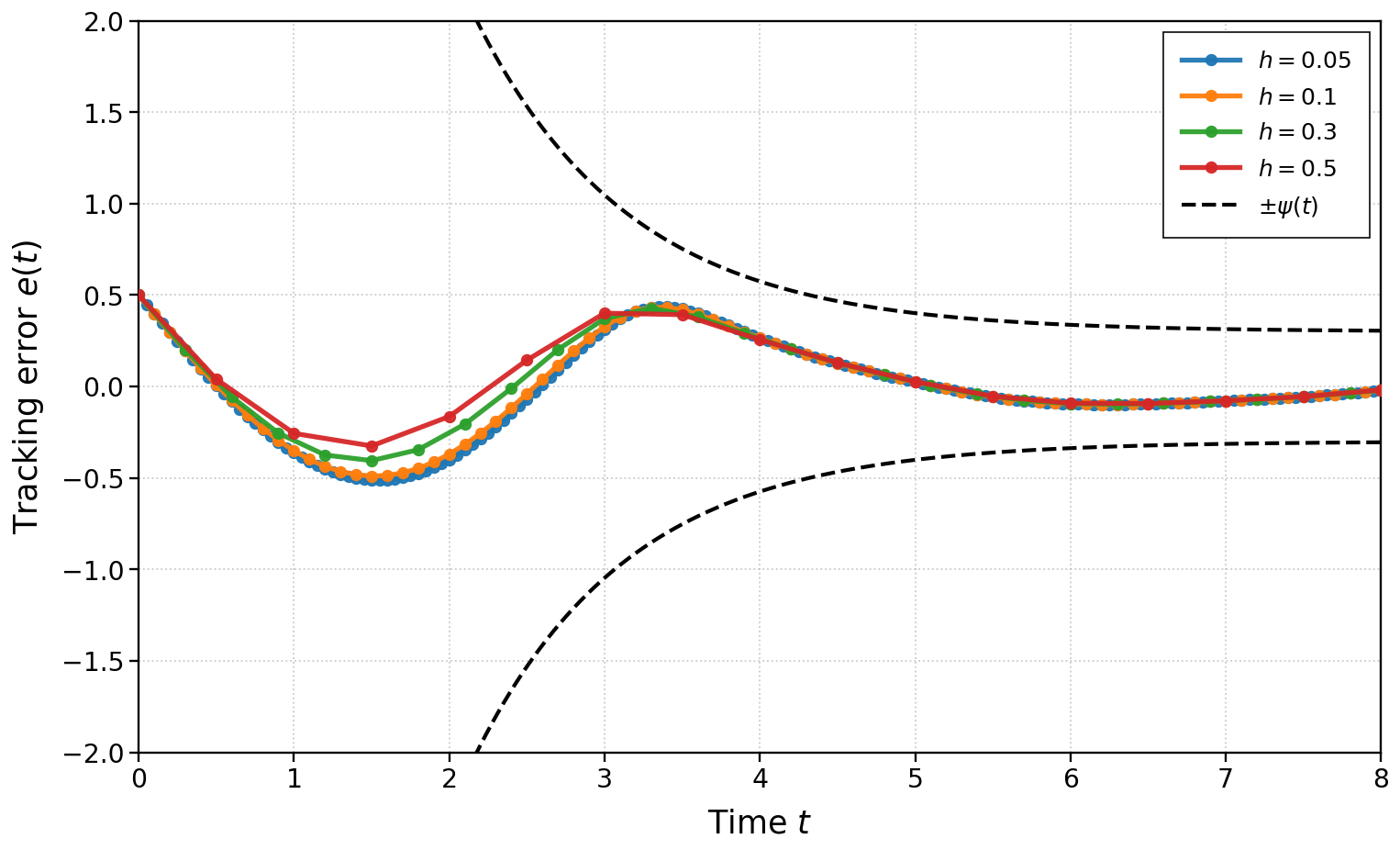}
    \caption{Simulation results for implicit Euler for different step sizes.}
    \label{fig:implicit}
\end{figure}
\begin{figure}
    \centering
    \includegraphics[width=0.9\linewidth]{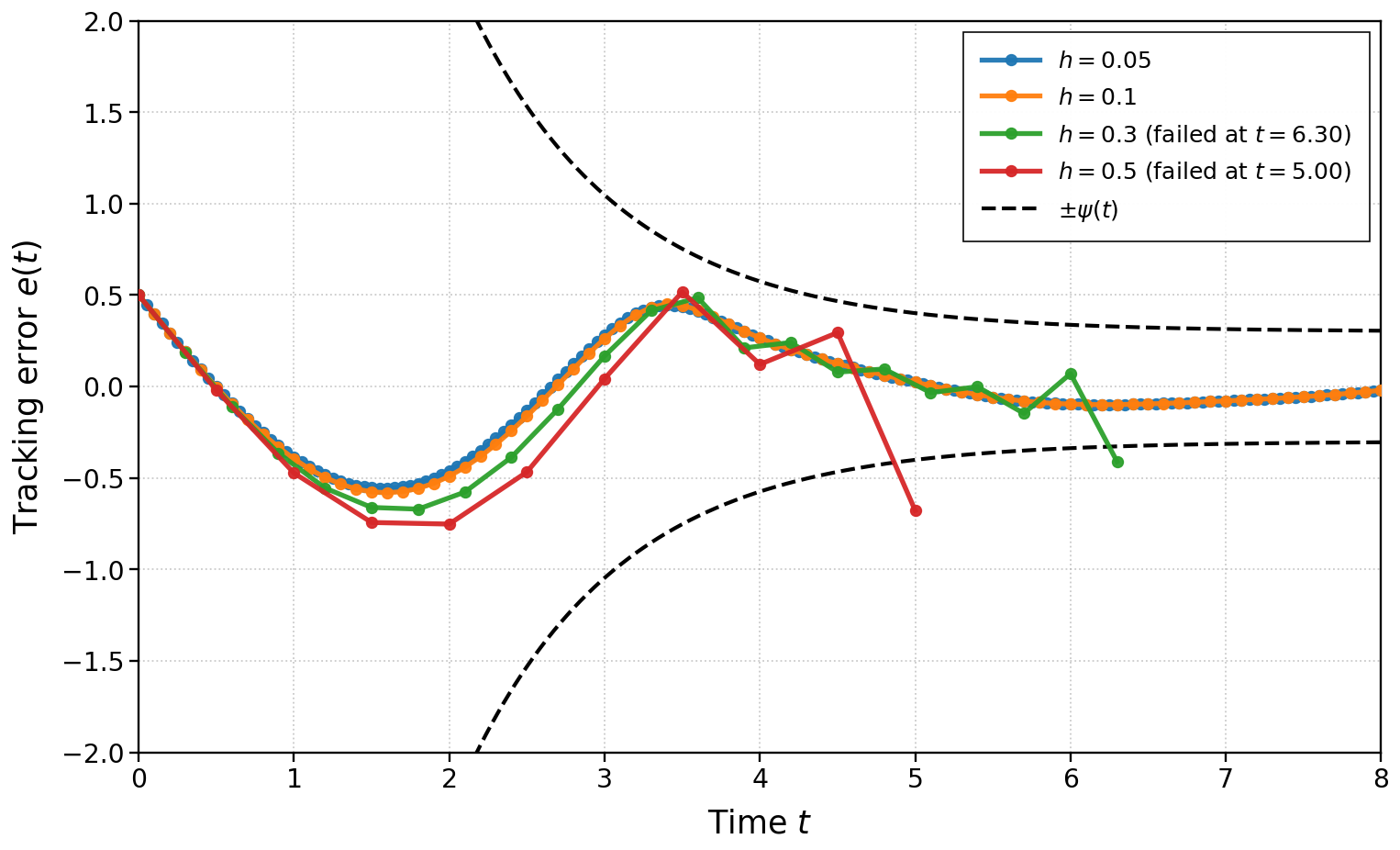}
    \caption{Simulation results for explicit Euler for different step sizes.}
    \label{fig:explicit}
\end{figure}
\section{Conclusion}\label{sec:conclusion}
We analyzed explicit and implicit Euler discretizations of the
continuous-time closed-loop dynamics generated by a funnel controller for a fully actuated nonlinear full-state system with scalar positive input gain. A direct scalar-gap argument first yielded a quantitative continuous-time funnel margin and the uniform regularity bounds needed for the error analysis.\\
For explicit Euler, the same gap mechanism gives a computable step-size condition under which all numerical values retain a prescribed nodal margin. The method converges with order one on finite horizons.  If the one-sided expansion of the plant is dominated by the funnel damping -- i.e., under
Assumption~\ref{ass:osl}, which bounds the plant's one-sided Lipschitz constant $\ell_f$ strictly below the minimal funnel damping rate
$g/\psi_+$ -- the forward map is contractive for sufficiently small steps and the first-order error estimate is uniform in time.\\
For implicit Euler, the singular radial term guarantees an admissible root for every step size and preserves the continuous-time margin at all nodes. Without Assumption~\ref{ass:osl}, sufficiently small steps already make the admissible root unique and yield first-order convergence on finite horizons. But under Assumption~\ref{ass:osl}, uniqueness holds for every step size, the
discrete flow is contractive in the sense of Lemma~\ref{lem:implicit-contractivity}, and the first-order error estimate is uniform in time.  The constants use only bounds on the plant, reference, funnel, and initial gap.

The analysis has some mathematical restrictions that suggest natural extensions. The scalar input gain $g>0$ was used essentially, both in the closed-form radial structure of $\gamma$ and in the strong-monotonicity estimate of
Lemma~\ref{lem:gamma-monotone}. Replacing $gI_m$ by a general positive definite high-frequency matrix $G$, or admitting relative degree greater than one, requires a separate, likely more involved, admissibility and error analysis.  Moreover, the results bound the error at the grid points
$t_k$ but say nothing about intersample behavior, which is a relevant question for implementation on physical hardware and connects to the sampled-data funnel-control literature~\cite{Lanza2024}. Finally, extending the discretization analysis to higher-order one-step or multistep methods, 
and to funnel controllers for systems of relative degree two or of
differential-algebraic type, are natural directions for future work.

\section*{Declaration on the use of AI-assisted tools}
OpenAI ChatGPT and Anthropic Claude were used to assist with language editing of the manuscript and refine the preliminary python scripts for the numerical simulations. All mathematical derivations, numerical results, and scientific conclusions are those of the authors. The authors independently verified all AI-assisted code and reviewed all AI-assisted content, and remain fully responsible for the content of the manuscript.

\section*{Funding}
Thomas Berger and Pushya Mitra acknowledge funding by the Deutsche Forschungsgemeinschaft (DFG, German
Research Foundation) -- Project-ID 362536361.

\bibliographystyle{elsarticle-num-names}
\bibliography{sample}

\end{document}